\documentclass[11pt]{article}

\usepackage{times}

\usepackage{fullpage}
\usepackage{amsthm}

\usepackage{enumitem}
\usepackage[colorlinks=true, allcolors=blue]{hyperref}

\usepackage{amssymb}
\usepackage[compact]{titlesec}
\usepackage{amsmath}
\usepackage{nicefrac}
\usepackage{mathdots,amsthm,url}
\usepackage{ifpdf,color}
\usepackage{graphicx}
\usepackage{here}
\usepackage{subfigure}
\usepackage{accents}
\usepackage{bbm}

\usepackage[ruled,vlined,linesnumbered]{algorithm2e}

\newtheorem{theorem}{Theorem}[section]

\newtheorem{proposition}[theorem]{Proposition}
\newtheorem{lemma}[theorem]{Lemma}

\theoremstyle{definition}
\newtheorem{example}[theorem]{Example}
\newtheorem{definition}[theorem]{Definition}
\newtheorem{remark}[theorem]{Remark}

\newcommand{\Delete}[1]{}
\newcommand{\pend}{\hspace*{\fill} $\Box$}

\usepackage[ruled,vlined,linesnumbered]{algorithm2e}

\newcommand{\B}{\mathbf{B}}
\newcommand{\U}{\mathbf{U}}
\newcommand{\bL}{\mathbf{L}}

\newcommand{\cC}{\mathcal{C}}
\newcommand{\cX}{\mathcal{X}}
\newcommand{\F}{\mathbf{F}}

\newcommand{\ctr}{\Psi}

\newcommand{\R}{\mathbb{R}}

\newcommand{\EE}{\mathcal F}

\newcommand{\dotp}[2]{\left\langle #1,#2\right\rangle}

\newcommand{\rk}{\textrm{rk}}

\newcommand{\supp}{\textrm{supp}}
\newcommand{\pr}[2]{\langle #1, #2 \rangle}

\newcommand{\In}[1]{I_0(#1)}
\newcommand{\Iu}[1]{I_1(#1)}
\newcommand{\J}[1]{J(#1)}

\newcommand{\Js}[1]{J^\star(#1)}

\newcommand{\Jss}{J^\star}

\newcommand{\gap}{\eta}

\DeclareMathOperator*{\argmin}{arg\,min}

\newif\ifnotes\notesfalse

\ifnotes
\usepackage{color}
\newcommand{\notename}[2]{{\textcolor{red}{\footnotesize{\bf (#1:} {#2}{\bf ) }}}}

\newcommand{\fnote}[1]{{\notename{Fujishige}{#1}}}
\newcommand{\knote}[1]{{\notename{Kitahara}{#1}}}
\newcommand{\lnote}[1]{{\notename{Laci}{#1}}}

\else

\newcommand{\notename}[2]{{}}

\newcommand{\fnote}[1]{}
\newcommand{\lnote}[1]{}
\newcommand{\knote}[1]{}

\fi

\newcommand\blfootnote[1]{%
  \begingroup
  \renewcommand\thefootnote{}\footnote{#1}%
  \addtocounter{footnote}{-1}%
  \endgroup
}

\begin{document}

\title{An Update-and-Stabilize Framework \\for the 
Minimum-Norm-Point Problem\blfootnote{
An extended abstract of this paper has appeared in Proceedings of the 24rd Conference on Integer Programming and Combinatorial Optimization, IPCO 2023.
    SF's research is supported by JSPS KAKENHI Grant Numbers JP19K11839 
and 22K11922 and  by the Research Institute for Mathematical
Sciences, an International Joint Usage/Research Center located in Kyoto
University. TK is supported by JSPS KAKENHI Grant Number JP19K11830. LAV's research is supported by the European Research Council (ERC) under the European Union’s Horizon 2020 research and innovation programme (grant agreement no. 757481--ScaleOpt).
}}

\author{{Satoru Fujishige}\footnote{
Research Institute for Mathematical Sciences,
Kyoto University, Kyoto 606-8502, Japan. 
E-mail: fujishig@kurims.kyoto-u.ac.jp}, 
\ \ {Tomonari Kitahara}\footnote{
Faculty of Economics, Kyushu University, Fukuoka 819-0395, Japan. 
E-mail: tomonari.kitahara@econ.kyushu-u.ac.jp},
\ \ and \ \ {L\'{a}szl\'{o} A.~V\'{e}gh}\footnote{
Department of Mathematics, London School of Economics and 
Political Science, London, WC2A 2AE, UK. 
\ \ E-mail: L.Vegh@lse.ac.uk}
}

\date{}

\maketitle

\begin{abstract}
We consider the  minimum-norm-point (MNP) problem over polyhedra, a well-studied problem that encompasses linear programming. We present a  general algorithmic framework that combines  two fundamental approaches for this problem: active set methods and first order methods. Our algorithm performs first order update steps, followed by iterations that aim to `stabilize' the current iterate  with additional projections, i.e., find a locally optimal solution whilst keeping the current  tight inequalities. Such steps have been previously used in active set methods for the nonnegative least squares (NNLS) problem.

We bound on the number of iterations polynomially in the dimension and in the associated circuit imbalance measure. 
In particular, the algorithm is strongly polynomial for network flow instances. Classical NNLS algorithms such as the Lawson--Hanson algorithm are special instantiations of our framework; as a consequence, we obtain convergence bounds for these algorithms. Our preliminary computational experiments show promising practical performance.
\end{abstract}



\section{Introduction}\label{sec:Introduction}


We study the minimum-norm-point (MNP) problem 
\begin{equation}\label{eq:min-norm-1}\tag{P}
\text{Minimize }\tfrac{1}{2}||Ax-b||^2\text{ subject to }{\mathbf{0}}\le x\le u\, ,\, x\in\R^N\, , 
\end{equation}
where $m$ and $n$ are positive integers,
$M=\{1,\cdots,m\}$ and $N=\{1,\cdots,n\}$, 
 $A\in \R^{M\times N}$ is a matrix with rank $\rk(A)=m$,  $b\in\R^M$, and $u\in(\R\cup\{\infty\})^N$. We will use the notation $\B(u):=\{x\in\R^N\mid {\mathbf{0}}\le x\le u\}$ for the feasible set.
 The problem \eqref{eq:min-norm-1} generalizes the linear programming (LP) feasibility problem: the optimum value is 0 if and only if $Ax=b$, $x\in\B(u)$ is feasible. The case $u(i)=\infty$ for all $i\in N$ is also known as the \emph{nonnegative least squares (NNLS) problem}, a fundamental problem in numerical analysis.

Two extensively studied approaches for MNP and NNLS are \emph{active set methods} and \emph{first order methods}. An influential active set method was proposed by Lawson and Hanson \cite[Chapter 23]{lawson1995solving} in 1974. Variants of this algorithm were also proposed by Stoer \cite{stoer1971numerical}, Bj\"orck~\cite{bjorck1988direct}, Wilhelmsen~\cite{Wilhelmsen1976}, and Leichner, Dantzig, and Davis \cite{leichner1993strictly}.\footnote{While there are minor differences in the details, these are essentially the same algorithm. Henceforth, we refer only to the Lawson--Hanson algorithm for simplicity.} Closely related is Wolfe's  classical minimum-norm-point algorithm \cite{Wolfe76}. 
These are iterative methods that maintain a set of active variables fixed at the lower or upper bounds, and passive (inactive) variables. In the main update steps, these algorithms fix the active variables at the lower or upper bounds, and perform unconstrained optimization on the passive variables. Such update steps require solving systems of linear equations. In all these methods, the set of columns corresponding to passive variables is linearly independent. The combinatorial nature of these algorithms enables to show termination with an exact optimal solution in a finite number of iterations. However, obtaining subexponential convergence bounds for such active set algorithms has remained elusive; see Section~\ref{sec:further} for more work on NNLS and Wolfe's algorithm.

In the context of first order methods,
the formulation \eqref{eq:min-norm-1} belongs to a family of problems for which Necoara, Nesterov, and Glineur \cite{Necoara2019} showed linear convergence bounds. That is, the number of iterations needed 
to find an $\varepsilon$-approximate solution depends linearly on 
$\log(1/\varepsilon)$. Such convergence has been known for strongly convex functions, but this property does not hold for \eqref{eq:min-norm-1}.
However,  \cite{Necoara2019}  shows that restricted variants of strong convexity also suffice for linear convergence. For problems of the form \eqref{eq:min-norm-1}, the required property follows using Hoffman-proximity bounds \cite{Hoffman52}; see \cite{Pena2020} and the references therein for recent results on Hoffman-proximity.
In contrast to active set methods, first order methods are computationally  cheaper as they do not require solving systems of linear equations. On the other hand, they do not find exact solutions.

\medskip

We propose a new algorithmic framework for the minimum-norm-point problem \eqref{eq:min-norm-1} that can be seen as a blend of active set and first order methods. Our algorithm 
performs \emph{stabilizing steps} between first order updates, and terminates with an exact optimal solution in a finite number of iterations.
Moreover, we show poly$(n,\kappa)$ running time bounds for multiple instantiations of the framework, where $\kappa$ is the \emph{circuit imbalance measure} associated with the matrix $(A\mid I_M)$ 
(see Section~\ref{sec:circuits}). 
This gives strongly polynomial bounds whenever $\kappa$ is constant; in particular, $\kappa=1$ for network flow feasibility. We note that if $A\in\mathbb{Z}^{M\times N}$, then $\kappa\le \Delta(A)$ for the maximum subdeterminant $\Delta(A)$. Still,  $\kappa$ can be exponential in the encoding length of the matrix.

The stabilizing step is similar to the one used by Bj\"orck \cite{bjorck1988direct} who considered the same formulation \eqref{eq:min-norm-1}. The Lawson--Hanson algorithm for the NNLS problem can be seen as special instantiations of our framework, and we obtain an $O(n^{2}m^2\cdot \kappa^2\cdot \|A\|^2\cdot\log(n+\kappa))$ iteration bound. 
These algorithms only use coordinate updates as first order steps, and maintain
 linear independence of the columns corresponding to passive variables.
Our framework is signficantly more general: we waive the linear independence requirement and allow for arbitarty active and passive sets.  This provides much additional flexibility, as our framework can be implemented with a variety of first order methods. This feature also yields a significant advantage in our computational experiments.

\paragraph{Overview of the algorithm}
A key concept in our algorithm is the \emph{centroid mapping},  defined as follows.
  For disjoint subsets $I_0,I_1\subseteq N$, we let $\bL(I_0,I_1)$ denote the affine subspace of  $\R^N$ 
where  $x(i)=0$ for $i\in I_0$ and $x(i)=u(i)$ for $i\in I_1$.
For $x\in\B(u)$, let $\In{x}$ and $\Iu{x}$ denote the 
subsets of coordinates $i$ with $x(i)=0$  and $x(i)=u(i)$, respectively. 
The \emph{centroid mapping} $\ctr: \B(u)\to \R^N$ is a mapping with the property
that $\ctr{(x)}\in \arg\min_y\{\tfrac{1}{2}\|Ay-b\|^2\mid y\in \bL(\In{x},\Iu{x})\}$. This mapping may not be unique, since the columns of $A$ corresponding to $\J{x}:=\{i\in N\mid 0<x(i)<u(i)\}$ may not be independent: the optimal \emph{centroid set} is itself an affine subspace. The point $x\in\B(u)$ is \emph{stable} if $\ctr(x)=x$.
This generalizes an update  used by Bj\"orck \cite{bjorck1988direct}. However, in his setting $J(x)$ is always linearly independent and thus the centroid set is always a single point. Stable sets can also be seen as the analogues of \emph{corral} solutions in Wolfe's minimum-norm point algorithm.

Every major cycle starts with an update step and ends with  a stable point.  The update step could be any first-order step satisfying some natural requirements, such as variants of Frank--Wolfe, projected gradient, or coordinate updates. As long as the current iterate is not optimal, this update strictly improves the objective. Finite convergence follows by the fact that there can be at most $3^n$ stable points.

After the update step, we start a sequence of minor cycles. From the current iterate $x\in \B(u)$, we move to $\ctr(x)$ in case $\ctr(x)\in \B(u)$, or to the intersection of the boundary of $\B(u)$ and the line segment $[x,\ctr(x)]$ otherwise.
The minor cycles finish once $x=\ctr(x)$ is  a stable point.
The objective $\tfrac{1}{2}\|Ax-b\|^2$ is decreasing in every minor cycle, and at least one new coordinate $i\in N$ is set to 0 or to $u(i)$. Thus, the number of minor cycles in any major cycle is at most $n$. One can use various centroid mappings satisfying a mild requirement on $\ctr$, described in Section~\ref{sec:MNPZ}.


We present a poly$(n,\kappa)$ convergence analysis for the NNLS problem with coordinate updates, which corresponds to the Lawson--Hanson algorithm and its variants.  We expect that similar arguments extend to the capacitated case. The proof has two key ingredients. First, we show linear convergence of the first-order update steps (Theorem~\ref{thm:geom-progress}). Such a bound follows already from \cite{Necoara2019}; we present a simple self-contained proof exploiting properties of stable points and the uncapacitated setting. The second step of the analysis shows that in every poly$(n,\kappa)$ iterations, we can identify a new variable that will never become zero in subsequent iterations (Theorem~\ref{thm:conv-main}).
The proof relies on proximity arguments: we show that for any iterate $x$ and any subsequent iterate $x'$, the distance $\|x-x'\|$ can be upper bounded in terms of $n$, $\kappa$, and the optimality gap at $x$.

In Section~\ref{sec:experiments}, we present preliminary computational experiments using randomly generated problem instances of various sizes. We compare the performance of different variants of our algorithm to standard gradient methods. 
For the choice of update steps, projected gradient performs much better than coordinate updates used in the NNLS algorithms. We compare an `oblivious' centroid mapping and one that chooses $\ctr(x)$ as the nearest point to $x$ in the centroid set in the \emph{`local norm'} (see Section~\ref{sec:1-1}). 
The latter one appears to be significantly better. For choices of parameters $n\ge 2m$, the running time of our method with projected gradient updates and local norm mapping
is typically within a factor two of \texttt{TNT-NN}, the state-of-the-art practical active set heuristic for NNLS \cite{myre2017tnt}, despite the fact that we only use simple linear algebra tools and have not made any attempts for practical speed ups. The performance is often better than  projected
 accelerated gradient descent, the best first order approach.

\paragraph{Proximity arguments and strongly polynomial algorithms}
Arguments that show strongly polynomial convergence by gradually revealing the support of an optimal solution are prevalent in combinatorial optimization. 
These  date back to Tardos's \cite{Tardos85} groundbreaking work giving the first strongly polynomial algorithm for minimum-cost flows. Our proof is closer to the dual `abundant arc' arguments by Fujishige \cite{Fujishige86} and Orlin \cite{Orlin93}.
Tardos generalized the above result for general LP's, giving a running time dependence poly$(n,\log \Delta(A))$, where $\Delta(A)$ is the largest subdeterminant of the constraint matrix. In particular, her algorithm is strongly polynomial as long as the entries of the matrix are polynomially bounded integers. This framework was recently strengthened in \cite{DadushNV20} to poly$(n,\log \kappa(A))$ running time for the circuit imbalance measure $\kappa(A)$. They also highlight the role of Hoffman-proximity and give such a bound in terms of $\kappa(A)$. 
We note that the above algorithms---along with many other strongly polynomial algorithms in combinatorial optimization---modify the problem directly once new information is learned about the optimal support. In contrast, our algorithm does not require any such modifications, nor a knowledge or estimate on the condition number $\kappa$.  Arguments about the optimal support only appear in the analysis.

Strongly polynomial algorithms with  poly$(n,\log \kappa(A))$ running time bounds can also be obtained using layered least squares interior point methods. This line of work was initiated by Vavasis and Ye \cite{Vavasis1996} using a related condition measure $\bar\chi(A)$. An improved version that also established the relation between $\bar\chi(A)$ and $\kappa(A)$ was recently given by Dadush et al.~\cite{DHNV20}. 
We refer the reader to the survey \cite{ENV22} for properties and further applications of circuit imbalances.

\subsection{Further related work}\label{sec:further} 
The Lawson--Hanson algorithm remains popular for the NNLS problem, and several variants are known.  Bro and De Jong \cite{bro1997fast}, and by Myre et al.~\cite{myre2017tnt} proposed empirically faster variants. In particular,  \cite{myre2017tnt}
allows bigger changes in the active and passive sets, thus waiving the linear independence on passive variables, and reports a significant speedup. However, there are no theoretical results underpinning the performance of these heuristics.

Wolfe's  minimum-norm-point algorithm \cite{Wolfe76} 
considers the variant of \eqref{eq:min-norm-1} where the box constraint $x\in\B(u)$ is replaced by $\sum_{i\in N} x_i=1$,   $x\ge \bf0$. 
It has been successfully employed 
as a subroutine in various optimization problems, e.g., 
submodular function minimization 
\cite{FujiIsotani2011}, see  also
\cite{Bach2013,Fuji80,FHYZ2009}. Beyond the trivial $2^n$ bound, the  convergence analysis remained elusive; the first bound with $1/\varepsilon$-dependence was given by Chakrabarty et al.~\cite{Chakrabarty2014} in 2014. Lacoste-Julien and Jaggi \cite{Lacoste2015} gave a $\log(1/\varepsilon)$ bound, parametrized by the \emph{pyramidal width} of the polyhedron.
Recently, De Loera et al.~\cite{DHR2020} showed an example of exponential 
time behaviour of Wolfe's algorithm for the \emph{min-norm insertion rule} (the analogue of a pivot rule); no exponential example for other insertion rules such as the  \emph{linopt}  rule used  in the application for submodular minimization.\footnote{The \emph{linopt} rule corresponds to the coordinate updates in the terminology of this paper.}

Our Update-and-Stabilize algorithm is also closely related to the Gradient Projection Method, see \cite{conn1988testing} and \cite[Section 16.7]{nocedal1999numerical}. 
This method also maintains a non-independent set of passive variables. For each gradient update, a more careful search is used in the gradient direction, `bending' the movement direction whenever a constraint is hit. The analogues of stabilizer steps are conjugate gradient iterations. Thus, this method avoids the computationally expensive step of exact projections; on the other hand, finite termination is not guaranteed.
We further discuss the relationship between the two algorithms  in Section~\ref{sec:concl}.

\medskip

There are similarities between our algorithm and the Iteratively Reweighted Least Squares (IRLS) method that has been intensively studied since the 1960's \cite{Lawson1961,Osborne1985}. For some $p\in[0,\infty]$, $A\in\R^{M\times N}$ and $b\in\R^M$, the goal is to approximately solve $\min\{\|x\|_p\mid\, Ax=b\}$. At each iteration, a weighted minimum-norm point $\min\{\dotp{w^{(t)}}{x}\mid\, Ax=b\}$ is solved, where the weights $w^{(t)}$ are iteratively updated. The LP-feasibility problem $Ax=b$, $\mathbf{0}\le x\le \bf{1}$ for finite upper bounds $u=\mathbf{1}$ can be phrased as an $\ell_{\infty}$-minimization problem  $\min\{\|x\|_\infty\mid\, Ax=b-A\mathbf{1}/2\}$. Ene and Vladu \cite{EneVladu2019} gave an efficient variant of IRLS for $\ell_1$ and $\ell_\infty$-minimization; see their paper for further references.
Some variants of our algorithm solve a weighted least squares problem with changing weights in the stabilizing steps.  There are, however significant differences between IRLS and our method. The underlying optimization problems are different, and IRLS does not find an exact optimal solution in finite time. Applied to LP in the $\ell_\infty$ formulation, IRLS satisfies $Ax=b$ throughout while violating the box constraints $\mathbf{0}\le x\le u$. In contrast, iterates of our algorithm violate $Ax=b$ but maintain $\mathbf{0}\le x\le u$. The role of the least squares subroutines is also rather different in the two settings.



\section{Preliminaries}
\label{sec:prelim}

\paragraph{Notation}
 We use $N\oplus M$ 
for disjoint union (or direct sum) 
of the copies of the two sets. 
For a matrix $A\in\R^{M\times N}$,
 $i\in M$ and $j\in N$, we denote the $i$th row of $A$ by $A_i$ and 
$j$th column by $A^j$.
Also for any matrix $X$ denote by $X^\top$ the matrix transpose of $X$.  
We let $\|\cdot\|_p$ denote the $\ell_p$ vector norm; we use $\|\cdot\|$ to denote the Euclidean norm $\|\cdot\|_2$.
For a matrix  $A\in\R^{M\times N}$, we let $\|A\|$ denote the spectral norm, that is, the $\ell_2\to\ell_2$ operator norm.

For any $x, y\in \R^M$ we define 
$\dotp{x}{y}=\sum_{i\in M}x(i)y(i)$. We will use this notation also in other dimensions. We let $[x,y]:=\{\lambda x+(1-\lambda)y\mid \lambda\in[0,1]\}$ denote the line segment between the vectors $x$ and $y$.

\subsection{Elementary vectors and circuits}\label{sec:circuits}

For a linear space $W\subsetneq \R^N$, $g\in W$  is an 
 \emph{elementary vector} if $g$ is a support minimal nonzero vector in $W$, that is, no $h\in W\setminus\{\bf0\}$ exists such that $\supp(h)\subsetneq \supp(g)$, where $\supp$ denotes the support of a vector. We let $\EE(W)\subseteq W$ 
 denote the set of elementary vectors.
 A \emph{circuit} in $W$ is the support of some elementary vector; these are precisely the circuits in the associated  linear matroid $\mathcal{M}(W)$.

The subspaces  $W=\{\bf0\}$ and $W=\R^N$ are called trivial subspaces; 
all other subspaces are nontrivial.
 We define the \emph{circuit imbalance measure}
 \[
\kappa(W):=\max\left\{\left|\frac{g(j)}{g(i)}\right|\mid g\in\EE(W), i,j\in \supp(g)\right\}\, 
\]
for nontrivial subspaces and $\kappa(W)=1$ for trivial subspaces.
For a matrix $A\in\R^{M\times N}$,  we  use the notation $\kappa(A)$ to denote  
$\kappa(\ker(A))$.

The following theorem shows the relation to totally unimodular (TU) matrices.
Recall that a matrix is \emph{totally unimodular (TU)} if the determinant of every square submatrix is $0$, $+1$, or $-1$.
\begin{theorem}[Cederbaum, 1957, see {\cite[Theorem 3.4]{ENV22}}]
  \label{thm:tu_iff_kappa_1}
  Let $W \subset \R^N$ be a linear subspace. Then $\kappa(W) = 1$ if and only if there exists a TU matrix $A\in\R^{M\times N}$ such that $W = \ker(A)$.
\end{theorem}
We also note that if $A\in\mathbb{Z}^{M\times N}$ is an integer matrix, then $\kappa(A)\le \Delta(A)$ for the maximum subdeterminant $\Delta(A)$.
\paragraph{Conformal circuit decompositions}
\label{par:sign_consistent}
We say that the vector $y \in \R^N$ \emph{conforms to}
$x\in\R^N$ if $x(i)y(i) > 0$ whenever $y(i)\neq 0$. 
Given a subspace $W\subseteq \R^N$, a \emph{conformal circuit decomposition} of a vector $v\in W$ is a decomposition
\[
v=\sum_{k=1}^\ell  h^k,
\]
where
$\ell\le n$ and  $h^1,h^2,\ldots,h^\ell\in \EE(W)$ are elementary vectors that conform to
$v$. A fundamental result on elementary vectors asserts the existence of a conformal circuit decomposition; see e.g., \cite{Fulkerson1968,RockafellarTheEV}. Note that there may be multiple conformal circuit decompositions of a vector.

\begin{lemma} \label{lem:conformal}
For every subspace $W\subseteq \R^N$,  every $v\in W$ admits a conformal circuit decomposition. 
\end{lemma}

Given $A\in\R^{M\times N}$, we define the extended subspace $\cX_A\subset \R^{N\oplus M}$ as $\cX_A:=\ker(A\mid -I_M)$. Hence, for every  $v\in \R^N$, $(v,Av)\in\cX_A$. For $v\in \R^N$, the \emph{generalized path-circuit decomposition of $v$ with respect to $A$} is a decomposition $v=\sum_{k=1}^\ell  h^k$,
where $\ell\le n$, and for each $1\le k\le \ell$, 
$(h^k,Ah^k)\in\R^{N\oplus M}$ 
is an elementary vector in $\cX_A$ that conforms to $(v,Av)$. Moreover, $h^k$ is an \emph{inner vector} in the decomposition if $Ah^k=\bf0$ 
and an \emph{outer vector} otherwise. 

We say that $v\in\R^N$ is  \emph{cycle-free with respect to $A$}, if all generalized path-circuit decompositions of $v$ contain outer vectors only. The following lemma will play a key role in analyzing our algorithms. 
\begin{lemma}\label{lem:contig-proximity}
For any $A\in\R^{M\times N}$, let $v\in\R^N$ be  cycle-free with respect to $A$. Then,
\[
 \|v\|_\infty\le \kappa({\cX_A})\cdot\|Av\|_1\, \quad\mbox{and}\quad \|v\|_2\le m\cdot\kappa({\cX_A})\cdot\|Av\|_2\, .
\]
\end{lemma}
\begin{proof}
Consider a generalized path-circuit decomposition $v=\sum_{k=1}^\ell h^k$. 
By assumption, $Ah^k\neq {\bf0}$ for each $k$. 
Thus,  for every $j\in\supp(h^k)$ there exists an $i\in M$, such that  $|h^k(j)|\le \kappa({\cX_A}) |A_i h^k|$. For every $j\in N$, the conformity of the decomposition implies $|v(j)|=\sum_{k=1}^\ell |h^k(j)|$. Similarly, for every $i\in M$,
$|A_i v|=\sum_{k=1}^\ell |A_i h^k|$. These imply the inequality $\|v\|_\infty\le \kappa({\cX_A}) \|Av\|_1$. 

For the second inequality, note that  for any outer vector $(h^k,Ah^k)\in \cX_A$, the columns in $\supp(h^k)$ must be linearly independent. Consequently,
$\|h^k\|_2\le \sqrt{m}\cdot \kappa({\cX_A})\cdot |(Ah^k)_i|$ for each $k$ and $i\in \supp(Ah^k)$. This implies 
\[
\|v\|_2\le \sum_{k=1}^\ell \|h^k\|_2\le \sqrt{m}\cdot \kappa({\cX_A})\cdot\|Av\|_1\le m\cdot \kappa({\cX_A})\cdot\|Av\|_2\, ,\]
completing the proof.
\end{proof}

\begin{remark}\label{remark:a-norm}
We note that a similar argument shows that $\|A\|\le \sqrt{m\tau(A)}\cdot\kappa(\cX_A)$, where $\tau(A)\le m$ is the maximum size of $\supp(Ah)$ for an elementary vector $(h,Ah)\in \cX_A$.
\end{remark}

\begin{example}
If $A\in\R^{M\times N}$ is the node-arc incidence matrix of a directed graph $D=(M,N)$. The system $Ax=b$, $x\in\B(u)$
corresponds to a network flow feasibility problem.  Here, $b(i)$ is the demand of node $i\in M$, i.e., the inflow minus the outflow at $i$ is required to be $b(i)$. Recall that $A$ is a TU matrix; consequently, $(A|-I_M)$ is also TU, and $\kappa({\cX_A})=1$. Our algorithm is strongly polynomial in this setting. Note that 
 inner vectors correspond to cycles and outer vectors to paths; this motivates the term `generalized path-circuit decomposition.' We also note $\tau(A)=2$, and thus $\|A\|\le \sqrt{2|M|}$ in this case.
 \end{example}

\subsection{Optimal solutions and proximity}\label{sec:1-1}

Let
\begin{equation}\label{eq:z1}
  Z(A,u):=\{Ax\mid x\in\B(u)\}.
\end{equation}
Thus, Problem \eqref{eq:min-norm-1} is
to find the point in $Z(A,u)$ that is nearest to $b$ with respect 
to the Euclidean norm. We note that if the upper bounds $u$ are finite, $Z(A,u)$ is  called a \emph{zonotope}.

Throughout, we let $p^*$ denote the optimum value of \eqref{eq:min-norm-1}.
Note that whereas the optimal solution $x^*$ may not be unique, the vector $b^*:=Ax^*$ is unique by strong convexity; we have $p^*=\tfrac{1}{2}\|b-b^*\|^2$. 
We use
\[
\gap(x):=\tfrac{1}{2}\|Ax-b\|^2-p^*
\]
to denote the optimality gap for $x\in\B(u)$. The point $x\in\B(u)$ is an \emph{$\varepsilon$-approximate solution} if $\gap(x)\le \varepsilon$.

For a point $x\in \B(u)$, let
\[
\In{x}:=\{i\in N\mid\, x(i)=0\}\, ,\quad \Iu{x}:=\{i\in N\mid\, x(i)=u(i)\}\, ,\quad\mbox{and}\quad \J{x}:=N\setminus (\In{x}\cup \Iu{x})\, .
\]
The gradient of the objective $\tfrac{1}{2}\|Ax-b\|^2$ in \eqref{eq:min-norm-1} 
 can be written as
\begin{equation}\label{eq:gradient}
g^x:=A^\top (Ax-b)\, .
\end{equation}
We recall the first order optimality conditions.
\begin{lemma}\label{lem:optimal}
The point $x\in \B(u)$ is an optimal solution to \eqref{eq:min-norm-1} if and only if $g^x(i)=0$ for all $i\in \J{x}$, $g^x(i)\ge 0$ for all $i\in \In{x}$, and $g^x(i)\le 0$ for all $i\in \Iu{x}$.
\end{lemma}
Using Lemma~\ref{lem:contig-proximity}, we can bound the distance of any $x$ from the nearest optimal solution.
\begin{lemma}\label{lem:opt-prox}
For any  $x\in \B(u)$, there exists an optimal solution $x^*$ to \eqref{eq:min-norm-1} such that
\[
\|x-x^*\|_\infty\le \kappa({\cX_A})\cdot\|Ax-b^*\|_1\, ,
\] 
and consequently,
\[
\|x-x^*\|_2\le m\cdot \kappa({\cX_A})\cdot\|Ax-b^*\|_2\, .
\]
\end{lemma}
\begin{proof}
Let us select an optimal solution $x^*$ to \eqref{eq:min-norm-1} such that $\|x-x^*\|$ is minimal. We show that $x-x^*$ is cycle-free 
w.r.t.~$A$; 
the statements then follow from Lemma~\ref{lem:contig-proximity}. 

For a contradiction, assume a generalized path-circuit decomposition of $x-x^*$ contains an inner vector $g$, i.e., $Ag=\bf0$. 
By conformity of the decomposition, for $\bar x=x^*+g$ we have $\bar x\in \B(u)$ and $A\bar x=A x^*$. Thus, $\bar x$ is another optimal solution, but $\|x-\bar x\|<\|x-x^*\|$, a contradiction.
\end{proof}
\subsection{The centroid mapping}\label{sec:MNPZ}


Let us denote by $3^N$ the set of all ordered pairs $(I_0,I_1)$ of 
disjoint subsets $I_0,I_1\subseteq N$, and let  $I_*:=\{i\in N\mid u(i)<\infty\}$. 
For any $(I_0,I_1)\in 3^N$ with $I_1\subseteq I_*$, we let
\begin{equation}\label{eq:affface1}
\bL(I_0,I_1):=\{x\in \R^N\mid \forall i\in I_0: x(i)=0,
\ \forall i\in I_1: x(i)=u(i)\  
\}\, .
\end{equation}

We call
$\{Ax \mid x\in \B(u)\cap \bL(I_0,I_1)\} \subseteq Z(A,u)$
a \emph{pseudoface} of  $Z(A,u)$. We note that every face of $Z(A,u)$ is a pseudoface, but there might be pseudofaces that do not correspond 
to any face.

We define a \emph{centroid set} for $(I_0,I_1)$ as
\begin{equation}\label{eq:ctr-def}
  \cC(I_0,I_1):=\arg\min_y\left\{ \|Ay-b\|\mid y\in \bL(I_0,I_1))\right\}\, .
\end{equation}
\begin{proposition} For $(I_0,I_1)\in 3^N$ with $I_1\subseteq I_*$, 
$\cC(I_0,I_1)$ is an affine subspace of $\R^N$, and   
for some $w\in \R^M$, it holds that $Ay=w$ for every $y\in\cC(I_0,I_1)$. 
\end{proposition}

The \emph{centroid mapping} $\ctr:\, \B(u)\to \R^N$ is a mapping that satisfies 
\[
\ctr(\ctr(x))=\ctr(x)\quad\mbox{and}\quad
\ctr(x)\in \cC(\In{x},\Iu{x})\, \quad \forall x\in\B(u)\, .
\]
We say that $x\in \B(u)$ is a \emph{stable point} if $\ctr(x)=x$.
A simple, `oblivious' centroid mapping arises by taking a minimum-norm point of the centroid set:
\begin{equation}\label{eq:pseudoinv}
\ctr(x):=\argmin\{\|y\|\mid y\in\cC(\In{x},\Iu{x})\}\, . 
\end{equation}
However, this mapping has some undesirable properties. For example, we may have an iterate $x$ that is already in $\cC(\In{x},\Iu{x})$, but $\ctr(x)\neq x$. 
Instead, we  aim for centroid mappings that move the current point `as little as possible'. This can be formalized as follows.
The centroid mapping $\ctr$ is called \emph{cycle-free}, if the vector
$\ctr(x)-x$ is cycle-free w.r.t.~$A$ 
for  every $x\in\B(u)$. The next claim describes a general class of cycle-free centroid mappings.
\begin{lemma}
For every $x\in \B(u)$, let  $D(x)\in\R^{N\times N}_{>0}$be a positive diagonal matrix. Then, 
\begin{equation}\label{eq:ctr-D-min}
\ctr(x):=\argmin\{\|D(x)(y-x)\|\mid y\in\cC(\In{x},\Iu{x})\}\,  
\end{equation}
defines a cycle-free centroid mapping.
\end{lemma}
\begin{proof}
For a contradiction, assume $y-x$ is not cycle-free for $y=\ctr{(x)}$, that is, a generalized path-circuit decomposition contains an inner vector $z$. For $y'=y-z$ we have $Ay'=Ay$, meaning that $y'\in \cC(\In{x},\Iu{x})$. This is a contradiction, since $\|D(x)(y'-x)\|<\|D(x)(y-x)\|$ for any positive diagonal matrix $D(x)$.
\end{proof}
We emphasize that $D(x)$ in the above statement is a function of $x$ and can be any positive diagonal matrix. Note also that the diagonal entries for indices in $\In{x}\cup\Iu{x}$ do not matter. In our experiments, defining $D(x)$ with diagonal entries $1/x(i)+1/(u(i)-x(i))$ for $i\in J(x)$ performs particularly well. Intuitively, this choice aims to move less the coordinates close to the boundary.\footnote{
Note that the weights for $i\in I_i(x)\cup I_1(x)$ do not matter, since we force $y(i)=x(i)$ on these coordinates. The choice $1/x(i)+1/(u(i)-x(i))$ would set $\infty$ on these coordinates.}
The next proposition follows from Lagrangian duality, and provides a way to compute $\ctr(x)$ as in \eqref{eq:ctr-D-min} by solving a system of linear equations.
\begin{proposition}\label{prop:centr}
For a partition $N=I_0\cup I_1\cup J$, the centroid set can be written as
\[
  \cC(I_0,I_1)=\left\{y\in \bL(I_0,I_1)\mid ({A^J})^\top(Ay-b)=\bf{0}\right\}\, .
  \]
  For $(I_0,I_1,J)=(\In{x},\Iu{x},\J{x})$ and $D=D(x)$,
the point $y=\ctr(x)$ as in \eqref{eq:ctr-D-min} can be obtained as the unique solution to the system of linear equations
\[
\begin{aligned}
(A^{J})^\top Ay&=(A^{J})^\top b&\\ 
y_{J}+(D_J^J)^{-1}(A^{J})^\top A^{J} \lambda&=x_{J}&\\
y(i)&=0&\quad\forall i\in I_0\\
y(i)&=u(i)&\quad\forall i\in I_1\\
&\lambda\in\R^{J}&\, .
\end{aligned}
\]
\end{proposition}

\section{The Update-and-Stabilize Framework}\label{sec:new}
\SetKwFunction{Update}{Update}

Now we describe a general algorithmic framework  \textsc{MNPZ}$(A,b,u)$ for solving \eqref{eq:min-norm-1}, shown in Algorithm~\ref{alg:mnp2}.
 Similarly to Wolfe's MNP algorithm, the algorithm comprises major and minor cycles. We maintain a point $x\in \B(u)$, and $x$ is stable at the end of every major cycle.
 Each major cycle  starts by calling the subroutine $\Update(x)$; the only general requirement on this subroutine is as follows: 
\begin{enumerate}[label=(U\arabic*), ref=U\arabic*]
\item\label{prop:update-1}
for $y=\Update(x)$, $y=x$ if and only if $x$ is optimal to \eqref{eq:min-norm-1}, and $\|Ay-b\|<\|Ax-b\|$ otherwise, and
\item\label{prop:update-2} if $y\neq x$, then for any $\lambda\in [0,1)$, 
$z=\lambda y+(1-\lambda)x$ satisfies $\|Ay-b\|<\|Az-b\|$.
\end{enumerate}
Property~\eqref{prop:update-1} can be obtained from any first order algorithm; we introduce some important examples in Section~\ref{sec:update}. Property~\eqref{prop:update-2} might be violated if using a fixed step-length, which is a common choice. 
In order to guarantee \eqref{prop:update-2}, we can post-process the first order update: choose $y$ as the optimal point on the line segment $[x,y']$, where $y'$ is the update found by the fixed-step update.

The algorithm terminates in the first major cycle when $x=\Update(x)$.
Within each major cycle, the minor cycles repeatedly use the centroid mapping $\Psi$. As long as $w:=\Psi(x)\neq x$, i.e., $x$ is not stable, we set $x:=w$ if $w\in\B(u)$; otherwise, we set the next $x$ as the intersection of the line segment $[x,w]$ and the boundary of $\B(u)$. The requirement \eqref{prop:update-1} is already sufficient to show finite termination.

\begin{algorithm}[htb]
\caption{\textsc{MNPZ}$(A,b,u)$}\label{alg:mnp2}
\DontPrintSemicolon
\SetKwInOut{Input}{Input}\SetKwInOut{Output}{Output}
\Input{$A\in\R^{M\times N}$, $b\in \R^M$, $u\in(\R\cup\{\infty\})^N$} 
\Output{An optimal solution $x$ to \eqref{eq:min-norm-1}} 
$x\gets$initial point from $\B(u)$ ; \\
\Repeat{$x=\Update(x)$}{
$x\gets \Update(x)$ ; \tcp*{Major cycle}
$w\gets \Psi(x)$ ;\;
\While{$\Psi(x)\neq x$\tcp*{Minor cycle}}{
   $\alpha^*\gets\arg\max\{\alpha\in [0,1]\mid x+\alpha (w-x)\in \B(u)\}$ ;\;
    \label{eq:find-alpha}
   $x\gets x+\alpha^* (w-x) $ ;\;
   $w\gets\Psi(x)$ ;\;
}
$x\gets w$ ;
}
\Return{$x$}\;
\end{algorithm}
\begin{theorem}\label{thm:finite}
Consider any $\Update(x)$ subroutine 
that satisfies \eqref{prop:update-1} and any centroid mapping $\Psi$.
The algorithm \textsc{MNPZ}$(A,b,u)$ finds an optimal solution to \eqref{eq:min-norm-1} within $3^n$ major  cycles. 
 Every major cycle contains at most $n$ minor cycles.
\end{theorem}
\begin{proof}
Requirement \eqref{prop:update-1} guarantees that if the algorithm terminates, it returns an optimal solution. 
We claim that the same sets $(I_0,I_1)$ cannot appear as $(\In{x},\Iu{x})$ at the end of two different major cycles; this implies the bound on the number of major cycles. To see this, we note that for $x=\Psi(x)$, $x\in\cC(\In{x},\Iu{x})=\cC(I_0,I_1)$; thus, 
$\|Ax-b\|=\min\left\{\|Az-b\|\mid \ z\in \bL(I_0,I_1))\right\}$. 
By \eqref{prop:update-1}, $\|Ay-b\|<\|Ax-b\|$ at the beginning of every major cycle. Moreover, it follows from the definition of the centroid mapping that  $\|Ax-b\|$ is non-increasing in every minor cycle. 
To bound the number of minor cycles in a major cycle, note that the set $\In{x}\cup \Iu{x}\subseteq N$ is extended in every minor cycle.
\end{proof}

\subsection{The Update subroutine}\label{sec:update}
We can implement the $\Update(x)$ subroutine satisfying \eqref{prop:update-1} and \eqref{prop:update-2} using various first order methods for constrained optimization. 

Recall the gradient $g^x$ from \eqref{eq:gradient}; 
we use $g=g^x$ when $x$ is clear from the context. 
The following property of stable points can be compared to the optimality condition in Lemma~\ref{lem:optimal}.
\begin{lemma}\label{lem:grad-x}
If $x(=\ctr(x))$ is a stable point, then
 $g^x(j)=0$ for all $j\in \J{x}$.
\end{lemma}
\begin{proof}
This directly follows from Proposition~\ref{prop:centr} that 
asserts $(A^{J(x)})^\top (Ax-b)=\bf0$. 
\end{proof}
We now describe three classical options. We stress that the centroid mapping $\ctr$ can be chosen independently from the update step.
\paragraph{The Frank--Wolfe update}
The Frank--Wolfe or \emph{conditional gradient} method is applicable only in the case when $u(i)$ is finite for every $i\in N$. In every update step, we start by computing $\bar y$ as the minimizer of the linear objective $\dotp{g}{y}$ 
over $\B(u)$, that is, 
\begin{equation}\label{eq:bar-y}
\bar y \in\arg\min\{\dotp{g}{y}\mid y\in\B(u)\}\, .
\end{equation}
 We set $\Update(x):=x$ if $\dotp{g}{\bar y}=\dotp{g}{x}$, or $y=\Update(x)$ is selected so that $y$ minimizes $\tfrac{1}{2}\|Ay-b\|^2$ on the line segment $[x,\bar y]$.

Clearly, $\bar y(i)=0$ whenever $g(i)>0$, and $\bar y(i)=u(i)$ whenever $g(i)<0$. However, $\bar y(i)$ can be chosen arbitrarily if $g(i)=0$. In this case, we keep 
 $\bar y(i)=x(i)$; this will be significant to guarantee stability of solutions in the analysis.

\paragraph{The Projected Gradient update}
The projected gradient update moves in the opposite gradient direction to $\bar y:=x-\lambda g$ for some step-length $\lambda>0$, and obtains the output $y=\Update(x)$ as the projection $y$ of $\bar y$ to the box $\B(u)$.
This projection simply changes every negative coordinate to $0$ and every $\bar y(i)>u(i)$ to $y(i)=u(i)$.
To ensure \eqref{prop:update-2}, we can perform an additional step that replaces $y$ by the point $y'\in[x,y]$ that minimizes $\tfrac{1}{2}\|Ay'-b\|^2$.

Consider now an NNLS instance (i.e., $u(i)=\infty$ for all $i\in N$), and let $x$ be a stable point. Recall $\Iu{x}=\emptyset$ in the NNLS setting. Lemma~\ref{lem:grad-x} allows us to write the projected gradient update in the following simple form that also enables to use optimal line search.
Define 
\begin{equation}\label{eq:z-PG}
z^x(i):=\max\{-g^x(i),0\} ,
\end{equation}
and use $z=z^x$ when clear from the context. According to Lemma~\ref{lem:optimal}, $x$ is optimal to \eqref{eq:min-norm-1} if and only if $z=\bf0$. 
We use the optimal line search
\[
y:=\arg\min_y\left\{\tfrac{1}{2}\|Ay-b\|^2\mid  y=x+\lambda z, \lambda\ge 0\right\}\, .
\]
If $z\neq\bf0$, this can be written explicitly as
\begin{equation}\label{eq:optline-update-PG}
y:=x+\frac{\|z\|^2}{\|Az\|^2}z\, .
\end{equation}
To verify this formula, we note that $\|z\|^2=-\pr{g}{z}$, since for every $i\in N$ either $z(i)=0$ or $z(i)=-g(i)$.

\paragraph{Coordinate update}
Our third update rule is the one used in the Lawson--Hanson algorithm.
 Given a stable point $x\in\B(u)$, we select a coordinate $j\in N$ where either $j\in \In{x}$ and $g(j)<0$ or $j\in \Iu{x}$ and $g(j)>0$, and set $y$ such that $y(i)=x(i)$ if $i\neq j$, and $y(j)$ is chosen in $[0,u(j)]$ so that $\tfrac{1}{2}\|Ay-b\|^2$ is minimized.
As in the Lawson--Hanson algorithm, we can maintain basic solutions throughout.
\begin{lemma}\label{lem:wolfe-indep}
Assume $A^J$ is linearly independent for $J=\J{x}$. Then, 
$A^{J'}$ is also linearly independent for
$J'=\J{y}=J\cup\{j\}$, where $y=\Update(x)$.
\end{lemma}
\begin{proof}
For a contradiction, assume $A^j=A^Jw$ for some $w\in\R^J$. Then, 
\[
g(j)=(A^j)^\top (Ax-b)=w^\top (A^J)^\top (Ax-b)=0\, ,
\]
a contradiction.
\end{proof}

Let us start with $x=\bf0$, 
i.e., $\J{x}=\Iu{x}=\emptyset$, $\In{x}=N$. Then, $A^{\J{x}}$ remains linearly independent throughout. Hence, every stable solution $x$ is a basic solution to \eqref{eq:min-norm-1}. Note that whenever $A^{\J{x}}$ is linearly independent, $\cC(\In{x},\Iu{x})$ contains a single point, hence, $\ctr(x)$ is uniquely defined.

\medskip

Consider now the NNLS setting.
For $z$ as in \eqref{eq:z-PG}, let us return $y=x$ if $z=\bf0$. 
Otherwise, let $j\in\arg\max_k z(k)$; 
note that $j\in\In{x}$. Let 
\begin{equation}\label{eq:optline-update-W}
y(i):=\begin{cases}
x(i)&\mbox{if }i\in N\setminus\{j\}\, ,\\
\frac{z(i)}{\|A^i\|^2}&\mbox{if }i=j\, .
\end{cases}
\end{equation}

\medskip

The following lemma is immediate. In the NNLS setting, \eqref{prop:update-2} is guaranteed for  the updates  described above. For the general form with upper bounds, we can post-process as noted above to ensure \eqref{prop:update-2}.
\begin{lemma}
The Frank--Wolfe, projected gradient, and coordinate update rules all satisfy \eqref{prop:update-1} and \eqref{prop:update-2}.
\end{lemma}

\paragraph{Cycle-free update rules}
\begin{definition} We say that $\Update(x)$ is a \emph{cycle-free update rule}, if for every $x\in\B(u)$ and $y=\Update(x)$, 
 $x-y$ is cycle-free w.r.t.~$A$. 
\end{definition}

\begin{lemma}\label{lem:solid-update}
The Frank--Wolfe, projected gradient, and coordinate updates are all cycle-free.
\end{lemma}
\begin{proof}
Each of the three rules satisfy that  
for any $x\in\B(u)$ with gradient $g$ and $y=\Update(x)$, 
$y-x$ conforms to $-g$. We show that this implies the required property.

For a contradiction, assume that a generalized path-cycle decomposition of $y-x$ contains an inner vector $h$. 
Thus, $h\neq \bf0$, $Ah=\bf0$, 
and $h$ conforms to $-g$. Consequently, $\dotp{g}{h}< 0$.
Recalling the form of $g$ from \eqref{eq:gradient}, we get
\[
0>\dotp{g}{h}=\dotp{A^\top (Ax-b)}{h}=\dotp{Ax-b}{Ah}=0\, ,\
\]
a contradiction.
\end{proof}

\section{Analysis}
Our main goal is to show the following convergence bound. The proof will be given in Section~\ref{sec:overall}.  Recall that in an NNLS instance, all upper capacities are infinite. 
\begin{theorem}\label{thm:conv-main}
Consider an NNLS instance of \eqref{eq:min-norm-1}, and assume we use a cycle-free centroid mapping.
Algorithm~\ref{alg:mnp2} terminates with an optimal solution in $O(n\cdot m^2\cdot\kappa^2(\cX_A)\cdot\|A\|^2\cdot\log(n+\kappa(\cX_A)))$ major cycles using projected gradient updates \eqref{eq:optline-update-PG}, and in  $O(n^{2}m^2\cdot\kappa^2(\cX_A)\cdot\|A\|^2\cdot\log(n+\kappa(\cX_A)))$ major cycles using coordinate updates \eqref{eq:optline-update-PG}, when initialized with $x=\mathbf{0}$.
In both cases, the total number of minor cycles 
is  $O(n^{2}m^2\cdot\kappa^2(\cX_A)\cdot\|A\|^2\cdot\log(n+\kappa(\cX_A)))$.
\end{theorem}
\subsection{Proximity bounds}\label{sec:prox}
We show that if using a cycle-free update rule and a cycle-free centroid mapping, the movement of the iterates in Algorithm~\ref{alg:mnp2} can be bounded by the change in the objective value.
First, a nice property of the centroid set  is that the movement of $Ax$ directly relates to the decrease in the objective value. Namely,
\begin{lemma}\label{lem:centr-move}
For $x\in \B(u)$, let $y\in\cC(\In{x},\Iu{x})$. Then, 
\[
\|Ax-Ay\|^2=\|Ax-b\|^2-\|Ay-b\|^2\, .
\]
Consequently, if $\Psi$ is a cycle-free centroid mapping and $y=\Psi(x)$, then 
\[
\|x-y\|^2\le m^2\cdot\kappa^2(\cX_A)\cdot\left(\|Ax-b\|^2-\|Ay-b\|^2\right)\, .
\]
\end{lemma}
\begin{proof}
Let 
$J:=\J{x}$. 
Since $Ax-b=(Ax-Ay)+(Ay-b)$, the claim is equivalent to showing that
\[
\dotp{Ax-Ay}{Ay-b}=0\, .
\]
Noting that $Ax-Ay=A^Jx_J-A^J y_J$, we can write
\[
\dotp{Ax-Ay}{Ay-b}=(x_J-y_J)^\top (A^J)^\top (Ay-b)=0\, .
\]
Where the equality follows since 
$(A^J)^\top ( Ay-b)=\bf0$ by Proposition~\ref{prop:centr}. 
The second part follows from Lemma~\ref{lem:contig-proximity}.
\end{proof}

Next, let us consider the movement of $x$ during a call to $\Update(x)$. 

\begin{lemma}\label{lem:update-move}
Let $x\in\B(u)$ and $y=\Update(x)$. Then,
\[
\|Ax-Ay\|^2\le \|Ax-b\|^2-\|Ay-b\|^2\, .
\]
If using a cycle-free update rule,  we also have 
\[
\|x-y\|^2\le m^2\cdot \kappa^2(\cX_A)\cdot \left(\|Ax-b\|^2-\|Ay-b\|^2\right)\, .
\]
\end{lemma}
\begin{proof}
From property \eqref{prop:update-2}, it is immediate to see that 
$\dotp{Ay-b}{Ax-Ay}\ge 0$.
This implies the first claim. The second claim 
follows from the definition 
of a cycle-free update rule and Lemma~\ref{lem:contig-proximity}.
\end{proof}

\begin{lemma}\label{lem:t-step-convergence}
Let $x\in\B(u)$, and let $x'$ be an iterate obtained by 
consecutive $t$ major or minor updates 
of Algorithm~\ref{alg:mnp2} using a cycle-free update rule and 
a cycle-free centroid mapping, 
starting from $x$. Then,
\[
\|x-x'\|\le m\cdot\kappa(\cX_A)\cdot\sqrt{2t}\cdot\sqrt{\tfrac{1}{2}\|Ax-b\|^2-\tfrac{1}{2}\|Ax'-b\|^2}\, .
\]
\end{lemma}
\begin{proof} Let us consider the (major and minor cycle) iterates $x=x^{(k)},x^{(k+1)},\ldots,x^{(k+t)}=x'$. From 
the triangle inequality, and the arithmetic-quadratic means inequality,
\[
\|x-x'\|\le \sum_{j=1}^{t} \|x^{(k+j)}-x^{(k+j-1)}\|\le \\\sqrt{t\sum_{j=1}^{t} \|x^{(k+j)}-x^{(k+j-1)}\|^2}
\]
The statement then follows using the bounds in
Lemma~\ref{lem:centr-move} and Lemma~\ref{lem:update-move}.
\end{proof}

\subsection{Geometric convergence of the projected gradient  and coordinate updates}
We present a simple convergence analysis for the NNLS setting.
For the general capacitated setting, similar bounds should follow from \cite{Necoara2019}. Recall that $\gap(x)$ denotes the optimality gap at $x$.
\begin{theorem}\label{thm:geom-progress}
Consider an NNLS instance of \eqref{eq:min-norm-1},  and 
let $x\ge \bf0$ be a stable point. 
Then for $y=\Update(x)$ using the projected gradient update \eqref{eq:optline-update-PG}  we have
\[
\gap(y)\le \left(1-\frac{1}{2m^2\cdot\kappa^2({\cX_A})\cdot\|A\|^2}\right)\gap(x)\, .
\]
Using coordinate updates as in  \eqref{eq:optline-update-W},  we have
\[
\gap(y)\le \left(1-\frac{1}{2nm^2\cdot\kappa^2({\cX_A})\cdot\|A\|^2}\right)\gap(x)\, .
\]
Consequently,  either with projected gradient or with coordinate updates, after performing $O(nm^2\cdot\kappa^2({\cX_A})\cdot\|A\|^2)$ minor and major cycles from an iterate $x$, we obtain an iterate $x'$ with $\gap(x')\le \gap(x)/2$.
\end{theorem}
Let us formulate the update progress using optimal line search.
\begin{lemma}\label{lem:update-progress-PG}
For a stable point $x\ge\bf0$, 
the update \eqref{eq:optline-update-PG} satisfies
\[
\|Ax-b\|^2-\|Ay-b\|^2\ge \frac{\|z\|^2}{\|A\|^2}\, ,
\]
and the update \eqref{eq:optline-update-W} satisfies
\[
\|Ax-b\|^2-\|Ay-b\|^2= \frac{z(j)^2}{\|A^j\|^2}\, .
\]
\end{lemma}
\begin{proof}
For the update \eqref{eq:optline-update-PG} with stepsize $\lambda=\|z\|^2/\|Az\|^2$,
we have
\[
\begin{aligned}
\|Ay-b\|^2&=\|Ax-b\|^2+\lambda^2\|Az\|^2+2\lambda\pr{Ax-b}{Az}\\
&=\|Ax-b\|^2+\lambda^2\|Az\|^2+2\lambda\pr{g}{z}\\
&=\|Ax-b\|^2+\lambda^2\|Az\|^2-2\lambda\|z\|^2\\
&=\|Ax-b\|^2-\frac{\|z\|^4}{\|Az\|^2}\, ,
\end{aligned}
\]
where the third equality uses $\pr{g}{z}=-\|z\|^2$ noted previously. The statement follows by using $\|Az\|\le \|A\|\cdot\|z\|$.

The proof is similar for the update \eqref{eq:optline-update-W}. Here, $y=x+\lambda e_j$, where $e_j$ is 
the $j$th unit vector, 
and $\lambda =z(j)/\|A^j\|^2$.  The bound follows by noting that $\pr{Ax-b}{Ae_j}=\pr{g}{e_j}=-z(j)$.
\end{proof}

We now use Lemma~\ref{lem:opt-prox} to bound $\|z\|$.
\begin{lemma}\label{lem:pg-z-bound}
For a stable point $x\ge \bf0$ and the update direction $z=z^x$ 
as in \eqref{eq:z-PG}, we have
\[
\|z\|\ge \frac{\sqrt{\gap(x)}}{\sqrt{2}m\cdot\kappa(\cX_A)}\, .
\]
\end{lemma}
\begin{proof}
Let $x^*\ge \bf0$ be an optimal solution to \eqref{eq:min-norm-1} 
as in Lemma~\ref{lem:opt-prox}, and $b^*=Ax^*$.
Using convexity of $f(x):=\tfrac{1}{2}\|Ax-b\|^2$, 
\[
p^*=f(x^*)\ge f(x)+\pr{g}{x^*-x}\ge f(x)-\pr{z}{x^*-x}\, , 
\]
where the second inequality follows by noting that for each $i\in N$, 
either $z(i)=-g(i)$, or $z(i)=0$ and $g(i)(x^*(i)-x(i))\ge 0$. 
From the Cauchy-Schwarz inequality and Lemma~\ref{lem:opt-prox}, we get
\[
p^*\ge f(x)-\|z\|\cdot \|x^*-x\|\ge f(x)- m\cdot\kappa(\cX_A)\cdot \|Ax-b^*\|\cdot \|z\|\, ,
\]
that is,
\[
\|z\|\ge \frac{\gap(x)}{m\cdot\kappa(\cX_A)\cdot\|Ax-b^*\|}\, .
\]
The proof is complete by showing
\begin{equation}\label{eq:ps-bs}
2\gap(x)\ge \|Ax-b^*\|^2\, .
\end{equation}
Recalling that $\gap(x)=\tfrac{1}{2}\|Ax-b\|^2-\tfrac{1}{2}\|Ax^*-b\|^2$ and that $b^*=Ax^*$, this is equivalent to 
\[
\pr{Ax-Ax^*}{Ax^*-b}\ge 0\, .
\]
This can be further written as
\[
\pr{x-x^*}{g^{x^*}}\ge 0\, ,
\]
which is implied by the first order optimality condition at $x^*$. This proves \eqref{eq:ps-bs}, and hence the lemma follows.
\end{proof}

\begin{proof}[Proof of Theorem~\ref{thm:geom-progress}]
The proof for the bound in projected gradient updates is immediate from Lemma~\ref{lem:update-progress-PG} and Lemma~\ref{lem:pg-z-bound}. For coordinate updates, recall that $j$ is selected as the index of the largest component $z(j)$. Thus, $z(j)^2\ge \|z\|^2/n$, and $\|A^j\|\le \|A\|$.

For the second part, the statement follows for projected gradient updates by the first part and by noting that there are at most $n$ minor cycles in every major cycle.
For coordinate updates,  every major cycle adds  one component to $J(x)$ whereas every minor cycle removes at least one. Hence, the total number of minor cycles is at most $m$ plus the total number of major cycles.
\end{proof}

\subsection{Overall convergence bounds}
\label{sec:overall}
In this subsection, we prove Theorem~\ref{thm:conv-main}.
Using Lemma~\ref{lem:t-step-convergence} and Theorem~\ref{thm:geom-progress}, we can derive the following stronger proximity bound:

\begin{lemma}\label{lem:strong-prox}
Consider an NNLS instance of \eqref{eq:min-norm-1}.
Let $x\ge\bf0$ 
be an iterate of Algorithm~\ref{alg:mnp2}  using projected gradient or coordinate updates, and 
let $x'\ge\bf0$ be any  later iterate. 
Then, for a value 
\[
\Theta:=O(\sqrt{n}m^2\cdot\kappa^2(\cX_A)\cdot\|A\|)\, ,
\]
 we have
\[
\|x-x'\|\le \Theta\sqrt{\gap(x)}\, .
\]
\end{lemma}
\begin{proof}
According to Theorem~\ref{thm:geom-progress}, after 
$T:=O(nm^2\cdot\kappa^2({\cX_A})\cdot\|A\|^2)$ major and minor cycles, we get to an iterate $x''$ with $\gap(x'')\le \gap(x)/4$. Thus, Lemma~\ref{lem:t-step-convergence} gives
\[
\|x-x''\|\le m\cdot \sqrt{2T}\cdot\kappa(\cX_A)\cdot\sqrt{\gap(x)}\, .
\]
Let us now define $x^{(k)}$ as the iterate following $x$ after $Tk$ major and minor cycles; we let $x^{(0)}:=x$. By Theorem~\ref{thm:geom-progress}, $x^{(k)}\le \gap(x)/4^k$, and similarly as above, for each $k=0,1,2,\ldots$ we get 
\[
\|x^{(k)}-x^{(k+1)}\|\le m\cdot \sqrt{2T}\cdot\kappa(\cX_A)\cdot\frac{\sqrt{\gap(x)}}{2^k}\, .
\]
The above bound also holds for any iterate $x'$ between $x^{(k)}$ an $x^{(k+1)}$. 
Using these bounds and the triangle inequality, for any iterate $x'$ after $x$, we obtain 
\[
\|x-x'\|\le 2m\cdot \sqrt{2T}\cdot\kappa(\cX_A)\cdot\sqrt{\gap(x)}\, .
\]
This completes the proof.
\end{proof}
We need one more auxiliary lemma.
\begin{lemma}\label{lem:x-hat-x}
Consider  an NNLS instance of \eqref{eq:min-norm-1}, and  let $x\ge \bf{0}$ be a stable point.
Let $\hat x\ge \bf{0}$ such that for each $i\in N$, either $\hat x(i)=x(i)$, or $\hat x(i)=0<x(i)$. Then,
\[
\|A\hat x-b\|^2=\|Ax-b\|^2+\|A\hat x-Ax\|^2\, .
\]
\end{lemma}
\begin{proof}
The claim is equivalent to showing
\[
\pr{A\hat x-Ax}{Ax-b}=0\, .
\]
We can write $\pr{A\hat x-Ax}{Ax-b}=\pr{g^x}{\hat x-x}$. By assumption, $\hat x(i)-x(i)\neq 0$ only if $x(i)>0$, but in this case $g^x(i)=0$ by Lemma~\ref{lem:grad-x}.
\end{proof}

For the threshold $\Theta$ as in Lemma~\ref{lem:strong-prox} and 
for any $x\ge \bf0$, let us define 
\[
\Js{x}:=\left\{i\mid\, x(i)>\Theta\sqrt{\gap(x)}\right\}  \, .
\]
The following is immediate from Lemma~\ref{lem:strong-prox}.
\begin{lemma}\label{lem:I-J-s}
Consider  an NNLS instance of \eqref{eq:min-norm-1}.
Let $x\ge \bf0$ be an iterate of Algorithm~\ref{alg:mnp2} 
using projected gradient updates, and $x'\ge \bf0$ be any  later iterate. 
Then,
\[
\Js{x}\subseteq \J{x'}\, .
\]
\end{lemma}

We are ready to prove Theorem~\ref{thm:conv-main}. 
\begin{proof}[Proof of Theorem~\ref{thm:conv-main}]
At any point of the algorithm, let $\Jss$ denote the union of the sets $\Js{x}$ for all iterations thus far. 
Consider a stable iterate $x$ at the beginning of any major cycle, and let 
\[
\varepsilon:=\frac{\sqrt{\gap(x)}}{4n\cdot\Theta\cdot\|A\|}\, .
\]
 Theorem~\ref{thm:geom-progress} guarantees that within 
 $O(nm^{2}\cdot\kappa^2(\cX_A)\cdot\|A\|^2\cdot\log(n+\kappa(\cX_A)))$ major and minor cycles we  arrive at an iterate $x'$ such that $\sqrt{\gap(x')}<\varepsilon$.  We note that $\log(n+\kappa(\cX_A)+\|A\|)=O(\log(n+\kappa(\cX_A)))$ according to Remark~\ref{remark:a-norm}.
  We show that 
\begin{equation}\label{eq:x-xp-extend}
\Js{x'}\cap \In{x}\neq\emptyset\, .
\end{equation} 
From here, we can conclude that  $\Jss$ was extended between iterates $x$ and $x'$. This may happen at most $n$ times, leading to the claimed bound on the total number of major and minor cycles. Using Theorem~\ref{thm:geom-progress} we also obtain the respective bounds on the number of major cycles for the two different updates.

For a contradiction, assume that \eqref{eq:x-xp-extend} does not hold. Thus, for every $i\in \In{x}$, we have $x'(i)\le \Theta\varepsilon$.
Let us define $\hat x\in\R^N$ as
\[
\hat x(i):=\begin{cases}
0&\mbox{if }i\in \In{x}\, ,\\
x'(i)&\mbox{if }i\in \J{x}\, .
\end{cases}
\]
By the above assumption, $\|\hat x-x'\|_\infty\le \Theta\varepsilon$, and therefore $\|A\hat x-Ax'\|\le \sqrt{n}\Theta\|A\|\varepsilon$. From Lemma~\ref{lem:x-hat-x}, we can bound
\begin{equation}\label{eq:x-hat-x-bound}
\|A\hat x-b\|^2=\|A x'-b\|^2+\|A\hat x-Ax'\|^2 \le 2p^*+ 
(n\Theta^2\|A\|^2+2)\varepsilon^2\, . 
\end{equation}

Recall that since $x$ is a stable solution, 
\[
\|Ax-b\|=\min\left\{\|Ay-b\|:\, y\in\bL(\In{x},\emptyset)\right\}\, .
\]
Since $\hat x$ is a feasible solution to this program, it follows that $\|A\hat x-b\|^2\ge \|Ax-b\|^2$.
We get that
\[
2\gap(x)=\|Ax-b\|^2-2p^*\le \|A\hat x-b\|^2-2p^*\le (n\Theta^2\|A\|^2+2)\varepsilon^2\, ,
\]
in contradiction with the choice of $\varepsilon$.
\end{proof}

\section{Computational Experiments}\label{sec:experiments}
We give preliminary computational experiments of different versions of our algorithm, and compare them to standard gradient methods and existing NNLS implementations. The experiments were  programmed and executed by MATLAB version R2023a on a personal computer having  11th Gen Intel(R) Core(TM) i7-11370H @ 3.30GHz and 16GB of memory.

We considered two families of randomly generated  NNLS instances. In Appendix~\ref{app:cap}, we also present experiments for capacitated instances (finite $u(i)$ values).

We tested each combination of two update methods: Projected Gradient (PG), and coordinate (C); and two centroid mappings, the `oblivious' mapping  \eqref{eq:pseudoinv} and the `local norm' mapping \eqref{eq:ctr-D-min} with diagonal entries $1/x(i),~i\in N$. Recall that for coordinate updates and starting from $x=\mathbf{0}$, there is a unique centroid mapping by Lemma~\ref{lem:wolfe-indep}.

Our first benchmarks are  
the projected gradient (PG) and the projected fast (accelerated) gradient (PFG) methods. In contrast to our algorithms, these do not finitely terminate. We stopped the algorithms once they found a near-optimal solution within a certain accuracy threshold. 

Further, we also compare our algorithms against the standard MATLAB implementation of the Lawson--Hanson algorithm called \texttt{lsqnonneg}, and against the implementation \texttt{TNT-NN} from \cite{myre2017tnt}. We note that \texttt{lsqnonneg} and the coordinate update version of our algorithms are essentially the same.

\begin{table}[hbtp]
  \caption{Computation time (in sec) for uncapacitated rectangular instances}
  \label{t7}
  \centering
  \begin{tabular}{clllllll}
    \hline
$m$ & 100 & 200 & 300 & 400 & 500 & 500 & 500 \\
$n$ & 200 & 400 & 600 & 800 & 1000 & 2000 & 3000 \\
 \hline
PG+(5) & 0.06 & 0.50 & 1.77 & 5.09 & 11.92 & 20.07 & 50.23 \\
PG+(6) & 0.01 & 0.06 & 0.18 & 0.28 & 0.29 & 0.17 & 0.22 \\
C & 0.05 & 0.31 & 1.02 & 2.52 & 4.84 & 2.98 & 2.90 \\
\hline
PG & 0.84 & 5.34 & 16.16 (1)& 21.85 (1)& 24.93 (2)& 0.07 & 0.05 \\
PFG & 0.06 & 0.53 & 1.52 & 2.06 & 4.22 & 0.09 & 0.07 \\
lsqnonneg & 0.01 & 0.15 & 0.52 & 1.28 & 2.64 & 1.53 & 1.55 \\
TNT-NN & 0.01 & 0.04 & 0.08 & 0.17 & 0.29 & 0.25 & 0.64 \\
 \hline
  \end{tabular}
\end{table}

\begin{table}[hbtp]
  \caption{\# of major cycles for uncapacitated rectangular instances}
  \label{t8}
  \centering
  \begin{tabular}{clllllll}
    \hline
$m$ & 100 & 200 & 300 & 400 & 500 & 500 & 500 \\
$n$ & 200 & 400 & 600 & 800 & 1000 & 2000 & 3000  \\
 \hline
PG+(5) & 6.4 & 9.4 & 13.4 & 11.6 & 13.4 & 1.0 & 1.0 \\
PG+(6) & 2.0 & 3.0 & 3.0 & 2.4 & 1.6 & 1.0 & 1.0 \\
C & 128.8 & 267.4 & 404.0 & 536.6 & 664.4 & 520.8 & 501.8 \\
 \hline
  \end{tabular}
\end{table}

\begin{table}[hbtp]
  \caption{The total \# of minor cycles for uncapacitated rectangular instances}
  \label{t9}
  \centering
  \begin{tabular}{clllllll}
    \hline
$m$ & 100 & 200 & 300 & 400 & 500 & 500 & 500\\
$n$ & 200 & 400 & 600 & 800 & 1000 & 2000 & 3000\\
 \hline
PG+(5) & 201.8 & 462.6 & 727.8 & 1129.8 & 1619.4 & 1029.6 & 1517.4 \\
PG+(6) & 17.4 & 32.6 & 42.0 & 32.0 & 18.8 & 2.0 & 1.0 \\
C & 157.6 & 336.6 & 512.2 & 675.8 & 828.4 & 542.4 & 509.2 \\
 \hline
  \end{tabular}
\end{table}

\begin{table}[hbtp]
  \caption{Computation time (in sec) for uncapacitated near-square instances}
  \label{t10}
  \centering
\scalebox{0.7}{
  \begin{tabular}{cllllllllllll}
    \hline
$m$ & 1000 & 1000 & 1000 & 1000 & 1000 & 1000 & 1000 & 1000 & 1000 & 1000 & 1000 & 1000\\
$n$ & 1020 & 1020 & 1020 &1020 & 1050 & 1050 & 1050 & 1050 & 1100 & 1100 & 1100 & 1100\\
Status & I & F (0.1) & F (0.5) & F (1) & I & F (0.1) & F (0.5) & F(1) & I & F (0.1) & F (0.5) & F (1) \\
\hline
PG+(5) & 13.44 & 0.88 & 2.59 & 0.83 & 14.74 & 2.21 & 2.20 & 1.49 & 16.85 & 4.46 & 9.61 & 2.40 \\ 
PG+(6) & 18.67 & 0.83 & 3.39 & 0.24 & 18.99 & 0.93 & 0.95 & 0.25 & 18.81 & 1.03 & 1.05 & 0.25 \\ 
C & 5.02 & 0.13 & 2.92 & 34.75 & 5.33 & 0.14 & 3.24 & 36.28 & 6.13 & 0.15 & 3.68 & 35.88 \\ \hline
PG & 0.17 & 0.49 & 1.41 & 60.00 (5)& 0.20 & 0.59 & 1.61 & 60.00 (5)& 0.25 & 0.81 & 2.32 & 18.36 \\
PFG & 0.16 & 0.12 & 0.31 & 60.00 (5)& 0.19 & 0.12 & 0.34 & 60.00 (5)& 0.24 & 0.15 & 0.41 & 60.00 (5)\\
lsqno & 4.74 & 0.11 & 2.65 & 27.97 & 4.22 & 0.10 & 2.46 & 28.00 & 5.75 & 0.11 & 3.30 & 29.19 \\
TNT-NN & 0.18 & 0.10 & 0.17 & 0.41 & 0.16 & 0.08 & 0.23 & 0.52 & 0.19 & 0.09 & 0.24 & 0.62 \\
 \hline
  \end{tabular}
}
\end{table}

\begin{table}[hbtp]
  \caption{\# of major cycles for uncapacitated near-square instances}
  \label{t11}
  \centering
\scalebox{0.8}{
  \begin{tabular}{cllllllllllll}
    \hline
$m$ & 1000 & 1000 & 1000 & 1000 & 1000 & 1000 & 1000 & 1000 & 1000 & 1000 & 1000 & 1000\\
$n$ & 1020 & 1020 & 1020 &1020 & 1050 & 1050 & 1050 & 1050 & 1100 & 1100 & 1100 & 1100\\
Status & I & F (0.1) & F (0.5) & F (1) & I & F (0.1) & F (0.5) & F (1) & I & F (0.1) & F (0.5) & F (1) \\
 \hline
PG+(5) & 4.2 & 1.0 & 1.2 & 1.2 & 4.0 & 1.0 & 1.0 & 1.0 & 4.4 & 1.0 & 1.6 & 1.0 \\
PG+(6) & 4.2 & 1.0 & 1.2 & 1.0 & 4.2 & 1.0 & 1.0 & 1.0 & 4.2 & 1.0 & 1.0 & 1.0 \\
C & 511.2 & 98.4 & 406.0 & 1080.6 & 524.4 & 101.0 & 423.4 & 1096.2 & 553.6 & 105.8 & 446.6 & 1092.6 \\
\hline
  \end{tabular}
}
\end{table}

\begin{table}[hbtp]
  \caption{Total \# of minor cycles for uncapacitated near-square instances}
  \label{t12}
  \centering
  \scalebox{0.8}{
  \begin{tabular}{cllllllllllll}
    \hline
$m$ & 1000 & 1000 & 1000 & 1000 & 1000 & 1000 & 1000 & 1000 & 1000 & 1000 & 1000 & 1000\\
$n$ & 1020 & 1020 & 1020 &1020 & 1050 & 1050 & 1050 & 1050 & 1100 & 1100 & 1100 & 1100\\
Status & I & F (0.1) & F (0.5) & F (1) & I & F (0.1) & F (0.5) & F (1) & I & F (0.1) & F (0.5) & F (1) \\
 \hline
PG+(5) & 572.4 & 21.0 & 78.2 & 18.6 & 601.4 & 51.0 & 51.0 & 32.8 & 636.0 & 101.0 & 270.8 & 52.2 \\
PG+(6) & 571.6 & 14.6 & 76.8 & 3.2 & 570.8 & 15.8 & 16.0 & 3.2 & 557.4 & 16.4 & 17.0 & 3.0 \\
C & 512.4 & 97.4 & 411.4 & 1160.2 & 529.0 & 100.0 & 431.8 & 1191.4 & 557.2 & 104.8 & 459.2 & 1184.2 \\
\hline
  \end{tabular}
}
\end{table}

\paragraph{Generating instances} 
We generated two families of experiments.
In the \emph{rectangular} experiments $n\ge 2m$, and in the \emph{near-square} experiments $m\le n\le 1.1m$. 
In both cases, the entries of the $m\times n$ matrix $A$ were chosen independently uniformly at random from the interval $[-0.5,0.5]$. In the rectangular experiments, the entries of $b$ were also chosen independently uniformly at random from $[-0.5,0.5]$.
Thus, the underlying LP $Ax=b$, $x\ge \bf{0}$ may or may not be feasible.

For the near-square instances, 
 such a random choice of $b$ leads to infeasible instances with high probability. We used this method to generate infeasible instances. We also constructed families where the LP is feasible as follows. For a sparsity parameter $\chi\in (0,1]$, we sampled a subset $J\subseteq N$, adding each variable independently with probability $\chi$, and generated coefficients $\{z_i: i\in J\}$ independently at random from $[0,1]$. We then set $b=\sum_{j\in J}A^j z_j$. 

\paragraph{Computational results}
We stopped each algorithm when the computation time reached 60 seconds. 
For each $(m,n)$, we test all the algorithms 5 times and the results shown here are 
the 5-run averaged figures. 

Table~\ref{t7} shows the overall computational times for rectangular instances; 
values in brackets show 
the number of trials whose computation time exceeded 60 seconds.
Tables~\ref{t8} and \ref{t9} show the number of major cycles, and the total number of minor cycles, respectively. 
Table~\ref{t10} shows the overall computational times for near-square instances. The status `I' denotes infeasible instances and `F' feasible instances, with the sparsity parameter $\chi$ in brackets, with values 0.1, 0.5, and 1.
Tables~\ref{t11} and \ref{t12} show the number of major cycles, and the total number of minor cycles, respectively, for near-square instances.

\paragraph{Comparison of the results}
For rectangular instances, 
 the `local-norm' update \eqref{eq:ctr-D-min} performs significantly better than the `oblivious' update \eqref{eq:pseudoinv}.
The `oblivious' updates are also outperformed by the coordinate updates, both in terms of running time as well as in the total number of minor cycles.

As noted above, while our algorithm with coordinate updates and  \texttt{lsqnonneg} are basically the same, the running time of the latter algorithm is better by around factor two. This is since \texttt{lsqnonneg} might be using more efficient linear algebra operations, in contrast to our more basic implementation.

The algorithm \texttt{TNT-NN} from \cite{myre2017tnt} is a fast practical algorithm using a number of heuristics, representing the state-of-the-art active set method for NNLS. 
Notably, our algorithm with `local-norm' updates \eqref{eq:ctr-D-min} is almost always within a factor two for rectangular instances, and performs better in some cases. This is despite the fact that we only use a basic implementation without using more efficient linear algebra methods or including further heuristics.

For rectangular instances, \texttt{TNT-NN}  and `local-norm' updates also  outperform fast projected gradient in most cases.

\medskip

The picture is more mixed for near-square instances. There is a marked difference between feasible and infeasible instances. The `local-norm' and `oblivious' update rules perform similarly, with a small number of major cycles. The number of minor cycles is much higher for infeasible instances. For infeasible instances, coordinate updates are faster than either variant of the PG update rule, while PG updates are faster for feasible instances.

The algorithm \texttt{TNT-NN} is consistently faster than our algorithm, with better running times for infeasible instances. For projected gradient and projected fast gradient, the running times are similar to  \texttt{TNT-NN} except for feasible instances with sparsity parameter $\chi=1$, where they do not terminate within the 60 seconds limit in most cases. In contrast, these appear to be the easiest instances to our method with PG updates with the `local-norm' mapping.

\section{Concluding Remarks}\label{sec:concl}

We have proposed a new `Update-and-Stabilize' framework for the minimum-norm-point problem \eqref{eq:min-norm-1}. Our method combines classical first order methods with `stabilizing' steps using the centroid mapping that amounts to computing a projection to an affine subspace. Our algorithm is always finite, and is strongly polynomial when the associated circuit imbalance measure is constant. In particular, this gives the first such convergence bound for the Lawson--Hanson algorithm.

There is scope for further improvements both in the theoretical analysis and in practical implementations. In this paper, we only analyzed the running time for uncapacitated instances. Combined with existing results from \cite{Necoara2019}, we expect that similar bounds can be shown for capacitated instances. We note that for the analysis, it would suffice to run minor cycles only once in a while, say after every $O(n)$ gradient updates. From a practical perspective however, running minor cycles after every update appears to be highly beneficial in most cases.
 Rigorous computational experiments, using standard families of LP benchmarks, is left for future work. 

Future work should also compare the performance of our algorithms to  the  
gradient projection method \cite{conn1988testing,nocedal1999numerical}, using techniques from that method to our algorithm and vice versa. We note that for NNLS instances, starting from a stable point our algorithm already finds the optimal gradient update. However,  a similar search as in gradient projection methods may be useful in the capacitated case.  In the other direction, we note that the conjugate gradient iterations used in gradient projection do not correspond to an explicit choice of a centroid mapping. A possible enhancement of gradient projection could come from approximating a  `local-norm' objective as in \eqref{eq:ctr-D-min} in the second stage.

We also point out that the `local-norm' selection rule \eqref{eq:ctr-D-min} was inspired by the affine scaling method; the important difference is that our algorithm moves all the way to the boundary, whereas affine scaling stays in the interior throughout.

\section*{Acknowledgments}
We are grateful to Andreas W\"achter for pointing us to the literature on the gradient projection method.
The third author would like to thank Richard Cole, Daniel Dadush, Christoph Hertrich, Bento Natura, and Yixin Tao for discussions on first order methods and circuit imbalances.

\bibliographystyle{abbrv}
\bibliography{fkv}

\begin{thebibliography}{10}

\bibitem{Bach2013}
F.~Bach.
\newblock Learning with submodular functions: A convex optimization
  perspective.
\newblock {\em Foundations and Trends in Machine Learning}, 6(2--3):145--373,
  2013.

\bibitem{bjorck1988direct}
{\AA}.~Bj{\"o}rck.
\newblock A direct method for sparse least squares problems with lower and
  upper bounds.
\newblock {\em Numerische Mathematik}, 54(1):19--32, 1988.

\bibitem{bro1997fast}
R.~Bro and S.~De~Jong.
\newblock A fast non-negativity-constrained least squares algorithm.
\newblock {\em Journal of Chemometrics: A Journal of the Chemometrics Society},
  11(5):393--401, 1997.

\bibitem{Chakrabarty2014}
D.~Chakrabarty, P.~Jain, and P.~Kothari.
\newblock Provable submodular minimization using wolfe's algorithm.
\newblock {\em Advances in Neural Information Processing Systems}, 27, 2014.

\bibitem{conn1988testing}
A.~R. Conn, N.~I. Gould, and P.~L. Toint.
\newblock Testing a class of methods for solving minimization problems with
  simple bounds on the variables.
\newblock {\em Mathematics of Computation}, 50(182):399--430, 1988.

\bibitem{DHNV20}
D.~Dadush, S.~Huiberts, B.~Natura, and L.~A. V{\'e}gh.
\newblock A scaling-invariant algorithm for linear programming whose running
  time depends only on the constraint matrix.
\newblock In {\em Proceedings of the 52nd Annual ACM Symposium on Theory of
  Computing (STOC)}, pages 761--774, 2020.

\bibitem{DadushNV20}
D.~Dadush, B.~Natura, and L.~A. V{\'{e}}gh.
\newblock Revisiting {T}ardos's framework for linear programming: {F}aster
  exact solutions using approximate solvers.
\newblock In {\em Proceedings of the 61st Annual IEEE Symposium on Foundations
  of Computer Science (FOCS)}, pages 931--942, 2020.

\bibitem{DHR2020}
J.~A. De~Loera, J.~Haddock, and L.~Rademacher.
\newblock The minimum {Euclidean}-norm point in a convex polytope: {Wolfe}'s
  combinatorial algorithm is exponential.
\newblock {\em SIAM Journal on Computing}, 49(1):138--169, 2020.

\bibitem{ENV22}
F.~Ekbatani, B.~Natura, and A.~L. V\'egh.
\newblock Circuit imbalance measures and linear programming.
\newblock In {\em Surveys in Combinatorics 2022}, London Mathematical Society
  Lecture Note Series, page 64–114. Cambridge University Press, 2022.

\bibitem{EneVladu2019}
A.~Ene and A.~Vladu.
\newblock Improved convergence for $\ell_1$ and $\ell_\infty$ regression via
  iteratively reweighted least squares.
\newblock In {\em International Conference on Machine Learning}, pages
  1794--1801. PMLR, 2019.

\bibitem{Fuji80}
S.~Fujishige.
\newblock Lexicographically optimal base of a polymatroid with respect to a
  weight vector.
\newblock {\em Mathematics of Operations Research}, 5(2):186--196, 1980.

\bibitem{Fujishige86}
S.~Fujishige.
\newblock A capacity-rounding algorithm for the minimum-cost circulation
  problem: A dual framework of the {Tardos} algorithm.
\newblock {\em Mathematical Programming}, 35(3):298--308, 1986.

\bibitem{FHYZ2009}
S.~Fujishige, T.~Hayashi, K.~Yamashita, and U.~Zimmermann.
\newblock Zonotopes and the {LP-Newton} method.
\newblock {\em Optimization and Engineering}, 10(2):193--205, 2009.

\bibitem{FujiIsotani2011}
S.~Fujishige and S.~Isotani.
\newblock A submodular function minimization algorithm based on the
  minimum-norm base.
\newblock {\em Pacific Journal of Optimization}, 7(1):3--17, 2011.

\bibitem{Fulkerson1968}
D.~Fulkerson.
\newblock Networks, frames, blocking systems.
\newblock {\em Mathematics of the Decision Sciences, Part I, Lectures in
  Applied Mathematics}, 2:303--334, 1968.

\bibitem{Hoffman52}
A.~J. Hoffman.
\newblock On approximate solutions of systems of linear inequalities.
\newblock {\em J. Res. Natl. Bur. Stand.}, 49(4):263--–265, 1952.

\bibitem{Lacoste2015}
S.~Lacoste-Julien and M.~Jaggi.
\newblock On the global linear convergence of {Frank--Wolfe} optimization
  variants.
\newblock {\em Advances in Neural Information Processing Systems}, 28, 2015.

\bibitem{Lawson1961}
C.~L. Lawson.
\newblock {\em Contribution to the theory of linear least maximum
  approximation}.
\newblock PhD thesis, 1961.

\bibitem{lawson1995solving}
C.~L. Lawson and R.~J. Hanson.
\newblock {\em Solving least squares problems}.
\newblock SIAM, 1995.

\bibitem{leichner1993strictly}
S.~Leichner, G.~Dantzig, and J.~Davis.
\newblock A strictly improving linear programming phase i algorithm.
\newblock {\em Annals of Operations Research}, 46:409--430, 1993.

\bibitem{myre2017tnt}
J.~M. Myre, E.~Frahm, D.~J. Lilja, and M.~O. Saar.
\newblock {TNT-NN}: a fast active set method for solving large non-negative
  least squares problems.
\newblock {\em Procedia Computer Science}, 108:755--764, 2017.

\bibitem{Necoara2019}
I.~Necoara, Y.~Nesterov, and F.~Glineur.
\newblock Linear convergence of first order methods for non-strongly convex
  optimization.
\newblock {\em Mathematical Programming}, 175(1):69--107, 2019.

\bibitem{nocedal1999numerical}
J.~Nocedal and S.~J. Wright.
\newblock {\em Numerical optimization}.
\newblock Springer, 1999.

\bibitem{Orlin93}
J.~B. Orlin.
\newblock A faster strongly polynomial minimum cost flow algorithm.
\newblock {\em Operations Research}, 41(2):338--350, 1993.

\bibitem{Osborne1985}
M.~R. Osborne.
\newblock {\em Finite algorithms in optimization and data analysis}.
\newblock John Wiley \& Sons, Inc., 1985.

\bibitem{Pena2020}
J.~Pe\~na, J.~C. Vera, and L.~F. Zuluaga.
\newblock New characterizations of {Hoffman} constants for systems of linear
  constraints.
\newblock {\em Mathematical Programming}, pages 1--31, 2020.

\bibitem{RockafellarTheEV}
R.~T. Rockafellar.
\newblock The elementary vectors of a subspace of {$R^N$}.
\newblock In {\em Combinatorial Mathematics and Its Applications: Proceedings
  North Carolina Conference, Chapel Hill, 1967}, pages 104--127. The University
  of North Carolina Press, 1969.

\bibitem{stoer1971numerical}
J.~Stoer.
\newblock On the numerical solution of constrained least-squares problems.
\newblock {\em SIAM Journal on Numerical Analysis}, 8(2):382--411, 1971.

\bibitem{Tardos85}
{\'E}.~Tardos.
\newblock A strongly polynomial minimum cost circulation algorithm.
\newblock {\em Combinatorica}, 5(3):247--255, Sep 1985.

\bibitem{Vavasis1996}
S.~A. Vavasis and Y.~Ye.
\newblock {A primal-dual interior point method whose running time depends only
  on the constraint matrix}.
\newblock {\em Mathematical Programming}, 74(1):79--120, 1996.

\bibitem{Wilhelmsen1976}
D.~R. Wilhelmsen.
\newblock A nearest point algorithm for convex polyhedral cones and
  applications to positive linear approximation.
\newblock {\em Mathematics of Computation}, 30(133):48--57, 1976.

\bibitem{Wolfe76}
P.~Wolfe.
\newblock Finding the nearest point in a polytope.
\newblock {\em Mathematical Programming}, 11(1):128--149, 1976.

\end{thebibliography}

\appendix
\section{Computational experiments for capacitated instances} \label{app:cap}

Tables~\ref{t1}--\ref{t6} show experimental results for capacitated instances. The instances were generated as in the NNLS case, with  upper capacities $u(i)=1$, $i\in N$. We did not use the benchmarks \texttt{lsqnonneg} and \texttt{TNT-NN} since these are not implemented for the capacitated setting. On the other hand, we also implemented our method with Frank--Wolfe updates, using both `local-norm' and `oblvious' centroid mappings. 
Among the first order benchmarks, we also included conditional gradient methods:  the Frank--Wolfe 
and away-step Frank Wolfe (AFW) methods.

In our framework, the Frank--Wolfe and projected gradient update rules performed similarly. In contrast, 
among the benchmark experiments, projected gradient methods consistently outperformed conditional gradient methods: the latter methods did not terminate within the 60 seconds limit for most cases. 

The overall experience is similar for uncapacitated (NNLS) and capacitated instances. Our method does well for rectangular instances, but is generally slower for infeasible near-square instances.

\begin{table}[hbtp]
  \caption{Computation time (in sec) for capacitated rectangular instances}
  \label{t1}
  \centering
\scalebox{0.9}{
  \begin{tabular}{clllllll}
    \hline
$m$ & 100 & 200 & 300 & 400 & 500 & 500 & 500\\
$n$ & 200 & 400 & 600 & 800 & 1000 & 2000 & 3000\\
 \hline
FW+(5) & 0.07 & 0.43 & 1.56 & 4.81 & 9.73 & 19.67 & 49.56 \\
FW+(6) & 0.02 & 0.07 & 0.22 & 0.49 & 0.63 & 0.17 & 0.31 \\
PG+(5) & 0.07 & 0.49 & 1.64 & 5.44 & 10.56 & 19.70 & 49.68 \\
PG+(6) & 0.01 & 0.07 & 0.20 & 0.44 & 0.55 & 0.17 & 0.32 \\
C & 0.05 & 0.28 & 0.94 & 2.19 & 5.35 & 2.88 & 2.88 \\ \hline
FW & 51.80 (4)& 54.78 (4)& 60.00 (5)& 60.00 (5)& 60.00 (5)& 0.18 & 0.20 \\
AFW & 51.76 (4)& 54.68 (4)& 60.00 (5)& 60.00 (5)& 60.00 (5)& 0.09 & 0.14 \\
PG & 0.31 & 2.64 & 5.25 & 5.44 & 33.33 (1)& 0.04 & 0.04 \\
PFG & 0.02 & 0.12 & 0.24 & 0.42 & 0.84 & 0.05 & 0.06 \\
 \hline
  \end{tabular}
}
\end{table}

\begin{table}[hbtp]
  \caption{\# of major cycles for capacitated rectangular instances}
  \label{t2}
  \centering
  \begin{tabular}{clllllll}
    \hline
$m$ & 100 & 200 & 300 & 400 & 500 & 500 & 500  \\
$n$ & 200 & 400 & 600 & 800 & 1000 & 2000 & 3000  \\
 \hline
FW+(5) & 6.2 & 7.8 & 10.2 & 10.2 & 12.8 & 1.0 & 1.0 \\
FW+(6) & 2.6 & 3.4 & 3.6 & 3.4 & 3.2 & 1.0 & 1.0 \\
PG+(5) & 6.8 & 9.4 & 11.8 & 11.4 & 14.4 & 1.0 & 1.0 \\
PG+(6) & 2.6 & 2.8 & 3.8 & 3.2 & 3.6 & 1.0 & 1.0 \\
C & 127.0 & 253.6 & 389.0 & 507.6 & 697.6 & 514.8 & 502.2 \\
 \hline
  \end{tabular}
\end{table}

\begin{table}[hbtp]
  \caption{The total \# of minor cycles for capacitated rectangular instances}
  \label{t3}
  \centering
  \begin{tabular}{clllllll}
    \hline
$m$ & 100 & 200 & 300 & 400 & 500 & 500 & 500  \\
$n$ & 200 & 400 & 600 & 800 & 1000 & 2000 & 3000 \\
 \hline
FW+(5) & 186.6 & 378.6 & 629.2 & 1066.0 & 1356.4 & 1016.6 & 1498.8 \\
FW+(6) & 28.2 & 42.4 & 52.2 & 61.8 & 43.4 & 2.0 & 2.0 \\
PG+(5) & 212.4 & 442.4 & 668.6 & 1213.0 & 1465.4 & 1013.8 & 1498.4 \\
PG+(6) & 27.2 & 37.4 & 48.0 & 55.8 & 38.8 & 2.0 & 2.0 \\
C & 155.6 & 313.2 & 483.8 & 622.6 & 897.0 & 532.2 & 511.2 \\
\hline
  \end{tabular}
\end{table}

\begin{table}[hbtp]
  \caption{Computation time (in sec) for capacitated near-square instances}
  \label{t4}
  \centering
\scalebox{0.5}{
  \begin{tabular}{cllllllllllll}
    \hline
$m$ & 1000 & 1000 & 1000 & 1000 & 1000 & 1000 & 1000 & 1000 & 1000 & 1000 & 1000 & 1000\\
$n$ & 1020 & 1020 & 1020 &1020 & 1050 & 1050 & 1050 & 1050 & 1100 & 1100 & 1100 & 1100\\
Status & I & F (0.1) & F (0.5) & F (1) & I & F (0.1) & F (0.5) & F (1) & I & F (0.1) & F (0.5) & F (1) \\
 \hline
FW+(5) & 12.95 & 0.85 & 0.85 & 1.71 & 14.62 & 2.11 & 7.80 & 2.65 & 16.55 & 6.47 & 13.19 & 3.33 \\
FW+(6) & 18.10 & 0.78 & 0.77 & 0.26 & 18.69 & 0.86 & 0.88 & 0.25 & 18.44 & 0.97 & 1.04 & 0.26 \\
PG+(5) & 12.67 & 0.82 & 0.85 & 1.71 & 14.23 & 2.04 & 7.62 & 2.85 & 16.38 & 4.24 & 12.62 & 3.36 \\
PG+(6) & 17.80 & 0.77 & 0.80 & 0.25 & 18.15 & 0.89 & 3.26 & 0.24 & 18.39 & 0.97 & 1.04 & 0.25 \\
C & 4.78 & 0.13 & 3.04 & 39.38 & 5.29 & 0.14 & 3.34 & 37.89 & 5.66 & 0.15 & 3.84 & 40.21 \\ \hline
FW & 60.00 (5)& 60.00 (5)& 60.00 (5)& 60.00 (5)& 60.00 (5)& 60.00 (5)& 60.00 (5)& 60.00 (5)& 60.00 (5)& 60.00 (5)& 60.00 (5)& 60.00 (5)\\
AFW & 60.00 (5)& 60.00 (5)& 60.00 (5)& 60.00 (5)& 60.00 (5)& 60.00 (5)& 60.00 (5)& 60.00 (5)& 60.00 (5)& 60.00 (5)& 60.00 (5)& 60.00 (5)\\
PG & 0.11 & 0.49 & 1.28 & 60.00 (5)& 0.14 & 0.52 & 1.84 & 60.00 (5)& 0.17 & 0.69 & 2.22 & 17.30 \\
PFG & 0.10 & 0.11 & 0.28 & 60.00 (5)& 0.11 & 0.11 & 0.37 & 60.00 (5)& 0.15 & 0.14 & 0.36 & 60.00 (5)\\
\hline
  \end{tabular}
}
\end{table}

\begin{table}[hbtp]
  \caption{\# of major cycles for capacitated near-square instances}
  \label{t5}
  \centering
\scalebox{0.8}{
  \begin{tabular}{cllllllllllll}
    \hline
$m$ & 1000 & 1000 & 1000 & 1000 & 1000 & 1000 & 1000 & 1000 & 1000 & 1000 & 1000 & 1000\\
$n$ & 1020 & 1020 & 1020 &1020 & 1050 & 1050 & 1050 & 1050 & 1100 & 1100 & 1100 & 1100\\
Status & I & F (0.1) & F (0.5) & F (1) & I & F (0.1) & F (0.5) & F (1) & I & F (0.1) & F (0.5) & F (1) \\
 \hline
FW+(5) & 4.4 & 1.0 & 1.0 & 2.8 & 4.0 & 1.0 & 2.0 & 1.8 & 4.4 & 1.2 & 2.4 & 1.0 \\
FW+(6) & 3.8 & 1.0 & 1.0 & 1.0 & 4.2 & 1.0 & 1.0 & 1.0 & 4.6 & 1.0 & 1.0 & 1.0 \\
PG+(5) & 3.8 & 1.0 & 1.0 & 2.8 & 3.8 & 1.0 & 1.8 & 2.0 & 4.2 & 1.0 & 2.2 & 1.0 \\
PG+(6) & 3.6 & 1.0 & 1.0 & 1.0 & 4.0 & 1.0 & 1.2 & 1.0 & 4.2 & 1.0 & 1.0 & 1.0 \\
C & 508.6 & 108.2 & 487.2 & 1387.2 & 527.0 & 112.8 & 513.6 & 1388.4 & 542.4 & 115.2 & 538.2 & 1446.8 \\
\hline
  \end{tabular}
}
\end{table}

\begin{table}[hbtp]
  \caption{Total \# of minor cycles for capacitated near-square instances}
  \label{t6}
  \centering
\scalebox{0.8}{
  \begin{tabular}{cllllllllllll}
    \hline
$m$ & 1000 & 1000 & 1000 & 1000 & 1000 & 1000 & 1000 & 1000 & 1000 & 1000 & 1000 & 1000\\
$n$ & 1020 & 1020 & 1020 &1020 & 1050 & 1050 & 1050 & 1050 & 1100 & 1100 & 1100 & 1100\\
Status & I & F (0.1) & F (0.5) & F (1) & I & F (0.1) & F (0.5) & F (1) & I & F (0.1) & F (0.5) & F (1) \\
 \hline
FW+(5) & 573.6 & 21.0 & 21.0 & 39.8 & 605.6 & 51.0 & 268.2 & 63.6 & 640.6 & 191.0 & 421.8 & 76.0 \\
FW+(6) & 573.6 & 14.2 & 14.0 & 3.4 & 576.6 & 15.0 & 15.2 & 3.2 & 561.4 & 16.0 & 17.2 & 3.2 \\
PG+(5) & 555.8 & 20.6 & 21.0 & 39.0 & 583.2 & 49.4 & 253.2 & 67.4 & 621.8 & 99.2 & 387.6 & 76.6 \\
PG+(6) & 553.2 & 14.6 & 14.6 & 3.4 & 554.0 & 15.8 & 72.2 & 3.2 & 548.8 & 16.4 & 17.0 & 3.2 \\
C & 510.4 & 106.6 & 500.2 & 1558.4 & 530.0 & 110.4 & 529.4 & 1571.8 & 545.8 & 113.2 & 558.2 & 1640.8 \\
\hline
  \end{tabular}
}
\end{table}

\end{document}